\documentclass[%
noccpublish,
  multilingual,german,english]{cc}

\usepackage{tilde}
\usepackage{maps}
\usepackage{mathell}
\usepackage{color}
\usepackage{pst-plot}

\newcommand{\F}{\mathbb{F}}

\newcommand{\lt}{\mathop{\text{\upshape lt}}}

\newcommand{\vars}{n}
\renewcommand{\vars}{\text{\textcolor{blue}{$n$}}}
\renewcommand{\vars}{r}
\newcommand{\dg}{d}
\renewcommand{\dg}{\text{\textcolor{blue}{$d$}}}
\renewcommand{\dg}{n}

\title{Counting decomposable multivariate polynomials}

\author{Joachim von~zur~Gathen\\
  B-IT\\
  Universit\"at Bonn\\
  D-53113 Bonn\\
  \email{gathen@bit.uni-bonn.de}\\
  \homepage{http://cosec.bit.uni-bonn.de/}
}

\begin{abstract}
  A polynomial $f$ (multivariate over a field) is \emph{decomposable}
  if $f=g \circ h$ with $g$ univariate of degree at least $2$. We
  determine the dimension (over an algebraically closed field) of the
  set of decomposables, and an approximation to their number over a
  finite field. The relative error in our approximations is
  exponentially decaying in the input size. 
\end{abstract}


\begin{document}

\section{Introduction}

It is intuitively clear that the decomposable polynomials form a small
minority among all polynomials (multivariate over a field). The goal
in this work is to give a precise quantitative version of this
intuition.  Interestingly, we find a special case for bivariate
polynomials where the intuition about the ``most general decomposable
polynomials'' is incorrect.

We use the methods from \cite{gat08-incl-gat07}, where the corresponding task was
solved for reducible, squareful, relatively irreducible, and singular
bivariate polynomials; further references are given in that
paper. \Citet*{gatvio09} extend those results to multivariate
polynomials and give further information such as exact formulas and
generating functions.

Our question has two facets: in the \emph{geometric} view, we want to
determine the dimension of the algebraic set of decomposable
polynomials, say over an algebraically closed field. The
\emph{combinatorial} task is to approximate the number of decomposables
over a finite field, together with a good relative error bound.  The
goal is to have a bound that is exponentially decreasing in the input
size. The choices we make in our calculations are guided by the goal
of such bounds in a form which is as simple and universal as possible.

As mentioned above, a special case occurs for bivariate polynomials.
Usually, the largest number of decompositions results from maximizing
the number of choices for the right component. But for some special
degrees---the squares of primes and numbers of RSA type---most
bivariate decompositions arise from having a large number of choices
for the left component. At three or more variables, all is uniform.

\cite{gie88} was the first to consider a variant of our counting
problem. He showed that the decomposable univariate polynomials form
an exponentially small fraction of all univariate polynomials. My
interest, dating back to the supervision of this thesis, was rekindled
by my study of similar counting problems 
\citep{gat08-incl-gat07},
and during a
visit to Pierre Dèbes' group at Lille, where I received a preliminary
version of \cite*{boddeb09}.

The companion paper \cite{gat08b} deals with decomposable univariate
polynomials.

\section{Decompositions}
\label{secDec}
We have a field $F$, a positive integer $\vars$, and the polynomial
ring $R = F[x_{1}, \ldots, x_{\vars}]$. We assume a degree-respecting term order on $R$,
so that in particular the \emph{leading term} $\lt(f)$ of an $f \in R$
is defined and $\deg \lt(f)= \deg f$. Throughout this paper, $\deg$
denotes the total degree.  If $f \neq 0$, the constant coefficient
$~lc(f) \in F^{\times} = F \smallsetminus \{0\}$ of $\lt(f)$ is the
\emph{leading coefficient} of $f$. Then $f$ is \emph{monic} if $~lc(f)
= 1$.
We call $f$ \emph{original} if its graph contains the origin, that is,
$f(0,\ldots,0) =0$.

The reader might think of the usual degree-lexicographic ordering,
where terms of higher degree come before those of lower degree, and
terms of the same degree are sorted lexicographically, with $x_{1} >
x_{2} > \cdots > x_{\vars}$. For example,
$$
f= -3x_{1}^{2}x_{3} - 2x_{2}^{3} + 4x_{4}x_{5}^{2} + 5x_{1}^{2} +
8x_{1}x_{2} + 5x_{6}^{2} -7
$$
is written in order, $~lc(f)=-3$ (provided that $-3 \neq 0$), and $f$
is not original (if $-7 \neq 0$). 

\begin{definition}
  \label{defComp}
  For $g\in F[t]$ and $h\in R$,
  $$
  f = g \circ h = g(h) \in R
  $$
  is their \emph{composition}.  If $\deg g \geq 2$ and $\deg h \geq
  1$, then $(g,h)$ is a \emph{decomposition} of $f$. A polynomial $f
  \in R$ is \emph{decomposable} if there exist such $g$ and $h$.
  Otherwise $f$ is \emph{indecomposable}. The decomposition $(g,h)$ is
  \emph{normal} if $h$ is monic and original. It is \emph{superlinear} if $\deg h
\geq 2$.
\end{definition}

There are other notions of decompositions. The present one is called
uni-multivariate in \cite{gatgut03-incl-gatgut99}. Another one is
studied in \cite{fauper08} for cryptanalytic purposes. In the context
of univariate polynomials, only superlinear decompositions are traditionally considered.

\begin{remark}
  \label{invariance1}
  Multiplication by a unit or addition of a constant does not change
  decomposability, since
  $$
  f = g \circ h \Longleftrightarrow a f+b = (a g+b) \circ h
  $$
  for all $f$, $g$, $h$ as above and $a,b \in F$ with $a\neq 0$.  In
  other words, the set of decomposable polynomials is invariant under
  this action of $F^{\times} \times F$ on $R$.

  Furthermore, any decomposition $(g,h)$ can be normalized by this
  action, by taking $a = ~lc (h)^{-1} \in F^{\times}$, $b=-a \cdot
  h(0,\ldots,0) \in F$, $g^{*} = g((t-b)a^{-1}) \in F[t]$, and $h^{*}
  = ah+b$.  Then $g\circ h = g^{*} \circ h^{*}$ and $(g^{*}, h^{*})$
  is normal.
\end{remark}

The following result is shown for $r \geq 2$ in \cite{boddeb09}.
It is trivially valid for $r=1$, where
\begin{equation}
  \label{trivialUni}
  f(x_{1})= f(t)\circ x_{1}
\end{equation}
for any $f \in F[x_{1}]$.

\begin{fact}
  \label{fact:uni}
  Any polynomial in $R$ has at most one normal decomposition with
  indecomposable right component.
\end{fact}

When the characteristic does not divide the degree of $f$, then this
also follows from the algorithmic approach in \cite{gat90c}, and also
holds for superlinear decompositions of univariate polynomials. If we
also allowed trivial decompositions $f = g \circ h$ with $\deg g = 1$,
then every polynomial would have exactly one normal decomposition with
indecomposable right component.


We fix some notation for the remainder of this paper. For $\vars \geq
1$ and $\dg \geq 0$, we write
$$
P_{\vars,\dg}= \{f \in F [x_{1}, \ldots, x_{\vars}] \colon \deg f \leq
\dg\}
$$
for the vector space of polynomials of degree at most $\dg$, of
dimension
$$
\dim P_{\vars,\dg} = b_{\vars,\dg}= \binom{\vars+\dg}{\vars}.
$$

Furthermore, we consider the subsets
\begin{align*}
  P_{\vars,\dg}^{=} &  = \{f \in P_{\vars,\dg} \colon \deg f = \dg\}, \\
  P^{0}_{\vars,\dg} & = \{f \in P_{\vars,\dg}^{=} \colon \text{$f$
    monic and original\}}.
\end{align*}

Over an infinite field, the first of these is the Zariski-open subset $
P_{\vars,\dg} \smallsetminus P_{\vars,\dg-1}$ of $P_{\vars,\dg}$ and
irreducible, taking $P_{\vars,-1} = \{0\}$. The second one is obtained
by further imposing one equation and working modulo multiplication by
units, so that
\begin{align*}
  \dim P_{\vars,\dg}^{=} &= b_{\vars,\dg},\\
  \dim P^{0}_{\vars,\dg}& =b_{\vars,\dg}{-2},
\end{align*}
with $P^{0}_{r,0}= \varnothing$. For any divisor $e$ of $\dg$, we have the normal composition map
$$
\map[\gamma_{\vars,\dg,e}] {P_{1,e}^{=} \times P^{0}_{\vars,\dg/e}}
{P_{\vars,\dg}^{=}} {(g,h)} {g \circ h,}
$$
corresponding to \ref{defComp}.  (Here $P_{1,e}^{=}$ consists of
polynomials in $F[t]$ rather than in $F[x_{1}]$.)  The set
$D_{\vars,\dg}$ of all decomposable polynomials in $P_{\vars,\dg}^{=}$
satisfies
\begin{equation}\label{gam}
  D_{\vars,\dg} = \bigcup_{1 < e\mid \dg} ~im \gamma_{\vars,\dg,e}.
\end{equation}

In particular, $D_{\vars,1} = \varnothing$ for all $\vars \geq 1$.
Over an algebraically closed field, the dimension of $D_{\vars,\dg}$
is taken to be the maximal dimension of its irreducible components.
We also call
$$
I_{\vars,\dg}= P_{\vars,\dg}^{=} \smallsetminus D_{\vars,\dg}
$$
the set of indecomposable polynomials.  Thus $I_{\vars,1} =
P_{\vars,1}^{=}$ for $\vars\geq 1$.

\begin{remark}
  \label{invariant}
  By \ref{invariance1}, over an algebraically closed field, the
  codimension of $D_{\vars,\dg}$ in $P_{\vars,\dg}^{=}$ equals that of
  $D_{\vars,\dg} \cap P_{\vars,\dg}^{0}$ in $P_{\vars,\dg}^{0}$.  The
  same holds for $I_{\vars,\dg}$, and over a finite field for the
  corresponding fractions.
\end{remark}

In order to have a nontrivial concept also in the univariate case,
where \ref{trivialUni} holds, we introduced in \ref{defComp} the
notion of {superlinear decompositions} $f = g \circ h$ where $\deg h
\geq 2$. The set of all these is
\begin{equation}\label{substack}
  D^{~sl}_{\vars,\dg}= \bigcup_{\substack{e\mid \dg\\1<e<\dg}}~im\gamma_{\vars,\dg,e}.
\end{equation}
In particular, $D^{~sl}_{\vars,\dg} = \varnothing$ if $\dg$ is
prime. We also let $I^{~sl}_{\vars, \dg}=P^{=}_{\vars, \dg} \smallsetminus
D^{~sl}_{\vars,\dg}.$
In the present paper, we investigate this notion only for two or more
variables. The univariate case is treated in \cite{gat08c}.

\section{Dimension of decomposables}\label{sec:dimofdec}

In this section, we determine the dimension of the set of decomposable
polynomials over an algebraically closed field. This forms the basis
for the counting result in the next section.

Throughout the paper, $l$ denotes the smallest prime factor of $\dg \geq
2$.  In the following, we have to single out the following special
case:
\begin{equation}
  \label{special}
  \text{$\vars = 2, \dg / l$ is prime and $\dg/l \leq 2l-5$}.
\end{equation}

The smallest examples are $\dg=l^{2}$ with $l \geq5$, $\dg=11 \cdot
13$, and $\dg=11 \cdot 17$.  In particular, $l$ and $\dg/l$ are always
at least $5$.

\begin{theorem}
  \label{thm:Dim}
  Let $F$ be an algebraically closed field, $\vars \geq 1$, $ \dg \geq
  2$, let $l$ be the smallest prime divisor of $\dg$, and
  \begin{align}\label{defM}
    m =
    \begin{cases}
      \dg & \text{if \ref{special} holds
        or}~ \vars=1,\\
      l  & \text{otherwise}.\\
    \end{cases}
  \end{align}
  Then the following hold.
  \begin{enumerate}
  \item\label{thm:subvar-1} $D_{\vars,\dg}$ has dimension
    $$
    \dim D_{\vars,\dg} = \binom{\vars+\dg/m}{\vars}+ m-1 .
    $$
  \item\label{thm:subvar-2} If $r\geq2$, then $I_{\vars, \dg}$ is a
    dense open subset of $P^{=}_{\vars, \dg}$, of dimension
    $\binom{\vars+\dg}{\vars}.$
  \item\label{thm:subvar-3} We assume that $r \geq2$. Then $
    D_{\vars,\dg}^{~sl} = \varnothing  $ if $\dg$ is prime, and otherwise
    \begin{align*}
      \dim D_{\vars,\dg}^{~sl}& = \binom {\vars+\dg/l}{\vars}+l-1.
    \end{align*}
  \end{enumerate}
\end{theorem}

\begin{proof}
  The claim \short\ref{thm:subvar-1} for $\vars=1$ follows from
  \ref{trivialUni}, and we assume $\vars\geq 2$ in the remainder of
  the proof.

\short\ref{thm:subvar-1} Each $\gamma_{\vars,\dg,e}$ is a
  polynomial map, and we have
  \begin{equation}\label{dim}
    \dim ~im \gamma_{\vars,\dg,e} \leq \dim P^{=}_{1,e} + \dim P^{0}_{\vars,\dg/e}
    = b_{\vars,\dg/e}+e - 1.
  \end{equation}
  We let $E=\{e \in \mathbb{N}\colon 1 < e\mid\dg \}$ be the index set in
  \ref{gam}.  When $\dg$ is prime, then $e = \dg= l$ is the only
  element of $E$, and the upper bound $\dim D_{\vars,\dg} \leq
  \vars+\dg$ in \short\ref{thm:subvar-1} follows.  We may now assume
  that $\dg$ is composite.  We consider the right hand side in
  \ref{dim} as the function
  \begin{equation}\label{real}
    u_{\vars,\dg}(e) = b_{\vars,\dg/e}+e - 1
  \end{equation}
  of a real variable $e$ on the interval $[1,\dg]$. See \ref{fig:u_example} for
  an example.
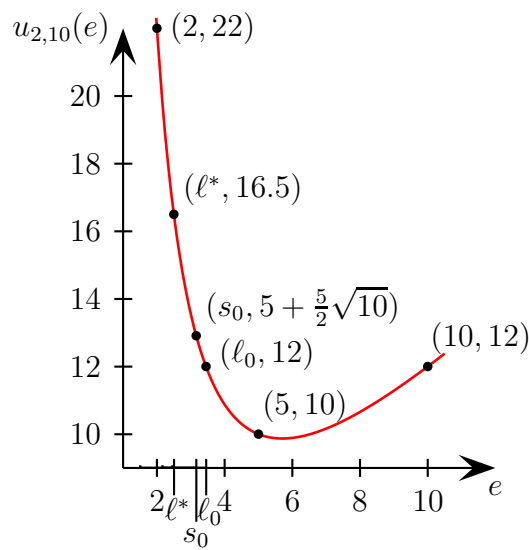
\begin{figure}
\begin{center}
\psset{unit=0.45cm}
\begin{pspicture*}(-2,6.65)(13.5,23.5)
\psaxes[showorigin=false, Dx=2, Ox=0, labels=x, ticks=x,
axesstyle=none]{->}(0,9)(11,10)
\psaxes[showorigin=false, ticks=x, labels=none, tickstyle=bottom,
ticksize=10pt]{}(1.5,9)(2.5,9)
\psaxes[showorigin=false, ticks=x, labels=none, tickstyle=top]{}(1.5,9)(2.5,9)
\uput[270](2.6,8.5){$l^{*}$}
\psaxes[showorigin=false, ticks=x, labels=none, tickstyle=bottom,
ticksize=20pt]{}(2.16,9)(3.16,9)
\psaxes[showorigin=false, ticks=x, labels=none, tickstyle=top]{}(2.16,9)(3.16,9)
\uput[270](3.16,7.6){$s_{0}$}
\psaxes[showorigin=false, ticks=x, labels=none, tickstyle=bootom,
ticksize=10pt]{}(2.45,9)(3.45,9)
\psaxes[showorigin=false, ticks=x, labels=none, tickstyle=top]{}(2.45,9)(3.45,9)
\uput[270](3.55,8.5){$l_{0}$}
\psline[arrowsize=10pt]{->}(1,9)(12,9)
\uput[270](12,9){$e$}   
\psaxes[showorigin=false, Dy=2, Oy=8, labels=y, ticks=y, axesstyle=none]{->}(1,8)(2,21)
\psline[arrowsize=10pt]{->}(1,9)(1,22)
\uput[180](1,22){$u_{2,10}(e)$}   
\psplot[linecolor=red, plotstyle=curve, linewidth=1pt]
{1.98}{10.5}{2 10 x div add 1 10 x div add mul 2 div x add 1 sub}
\psdot(2,22)
\uput[0](2,22){$(2,22)$}
\psdot(5,10)
\uput[50](5,10){$(5,10)$}
\psdot(10,12)
\uput[50](10,12){$(10,12)$}
\psdot(3.16,12.91)
\uput[45](3.16,12.91){$(s_{0},5+\frac{5}{2}\sqrt{10})$}
\psdot(3.45,12)
\uput[22.5](3.45,12){$(l_{0},12)$}
\psdot(2.5,16.5)
\uput[45](2.5,16.5){$(l^{*}, 16.5)$}
\end{pspicture*}
\end{center}
 \caption{An example of $u_{r,n}$, for $r=2$ and $n=10$, with $l=2$,
   $l^{*}=\frac{5}{2}$, $s_{0}=\sqrt{10}\approx 3.16$, and $l_{0}=1+\sqrt{6}\approx 3.45$.}
 \label{fig:u_example}
\end{figure}
We claim that
  \begin{equation}
    \label{maxU}
    u_{\vars,\dg}(m)
    = \max_{e\in E} u_{\vars,\dg} (e) .
  \end{equation}
  The upper bound in \short\ref{thm:subvar-1} follows from this.  The
  second derivative
  \begin{align*}
    \frac{\partial^{2}u_{\vars,\dg}}{\partial e^{2}}(e)=
    \frac{\dg}{e^{3}\cdot \vars!} \sum_{1 \leq i \leq \vars
    }\biggl(\frac{\dg}{e} \sum_{\substack{1 \leq j \leq \vars \\ j
        \neq i}} \prod_{\substack{1 \leq k \leq \vars \\ k \neq
        i,j}}(k+ \frac{\dg}{e})+2 \prod_{\substack{1 \leq j \leq
        \vars\\j \neq i}}(j+\frac{\dg}{e})\biggr)
  \end{align*}
  is positive on $[1,\dg]$, so that $u_{\vars,\dg}$ is convex. In
  particular, $u_{\vars,\dg}$ takes its maximum on the interval
  $[l,\dg]$ at one of the two endpoints.  

  For \ref{maxU}, we start with the case $\vars \geq 3$ and claim that
  $u_{\vars,\dg}(l) \geq u_{\vars,\dg}(\dg)$.
  Setting $s_{0}= \sqrt{\dg}$, we have
  $$
  u_{\vars,\dg}(s_{0})- u_{\vars,\dg}(\dg) =
  \binom{\vars+s_{0}}{\vars}+s_{0}-1-(\vars+s^{2}_{0}).
  $$ 
  Now we replace $s_{0}$ by a real variable $s$, and set
  \begin{align*}
    v_{\vars}(s)& = \binom{\vars+s}{\vars}+s-1-(\vars+s^{2}).
  \end{align*}
  Then
  \begin{equation}\label{pos}
    v_{\vars}(2)= (\vars^{2}+\vars-4)/2>0,
  \end{equation}
   since $\vars \geq 2$. Furthermore, we have
  $$
  \frac{\partial v_{\vars}}{\partial s}(s)= \frac{1}{\vars!} \sum_{1
    \leq i \leq \vars} \prod_{\substack{1 \leq j \leq \vars \\ j \neq
      i}} (j+s) +1 -2s.
  $$
  Expanding the product, we find that the coefficient in the sum of
  the linear term in $s$ equals
  \begin{align*}
    \sum_{1 \leq i \leq \vars} \sum_{\substack{1 \leq j \leq \vars \\
        j \neq i}} \prod_{\substack{1 \leq k \leq \vars \\ k \neq
        i,j}}k = \vars!  \sum_{\substack{1 \leq i,j \leq \vars \\j
        \neq i}} \frac{1}{i \cdot j} \geq \vars! \cdot 2 \cdot
    (\frac{1}{1 \cdot 2} + \frac{1}{1 \cdot 3} + \frac{1}{2 \cdot 3})
    = 2 \cdot \vars!,
  \end{align*}
  since $\vars \geq 3$.  Thus
  $$
  \frac{\partial v_{\vars}}{\partial s}(s) \geq 0,
  $$
  and together with \ref{pos} this implies $v_{\vars}(s) >0$ for all
  $s \geq 2$. Since $\dg$ is composite, we have $2 \leq l \leq
  \sqrt{\dg} = s_{0}< \dg$, and from the above we have
  \begin{align*}\label{compos}
   u_{r,\dg}(l)\geq u_{\vars,\dg}(s_{0}) \geq u_{\vars,\dg}(\dg). 
  \end{align*}
  Since $m=l$, this shows the claim \ref{maxU} and the upper bound in
  \short\ref{thm:subvar-1}.

  For the case $\vars=2$, we observe that
  \begin{equation}
    \label{ul-ud}
    u_{2,\dg}(l)-u_{2,\dg}(\dg)= \frac {(\dg-l)(\dg+4l-2l^{2})} {2l^{2}}
  \end{equation}
  is nonnegative if and only if $ l \leq l_{0}$, where $l_{0} = 1 +
  \frac{1}{2} \sqrt{2\dg+4}$ is the positive root of the quadratic
  factor. Furthermore, we note that
  \begin{equation}\label{eq:qf}
  u_{2,\dg}(\dg)> u_{2, \dg}(l) \Longleftrightarrow l > l_{0}
  \Longleftrightarrow \dg/l < 2l-4 \Longleftrightarrow \dg/l \leq
  2l-5,
  \end{equation}
  $$
  l_{0}^{2}= \dg/2+ \sqrt{2\dg+4} +2 > \dg/2.
  $$

  If the conditions in \ref{eq:qf} hold, there is at most one other
  prime factor of $\dg$ besides $l$, so that $\dg/l$ is
  prime and \ref{special} holds. \ref{maxU} follows in this case, and also
  otherwise because of the equivalences in \ref{eq:qf}.

  We have now shown one inequality in \short\ref{thm:subvar-1}, namely
  that $\dim D_{\vars,\dg} \leq u_{\vars,\dg}(m)$. For
  \short\ref{thm:subvar-2}, we claim that $u_{\vars,\dg}(m) <
  u_{\vars,\dg}(1) = \dim P_{\vars,\dg}^{=}$.  Since $1 < m \leq \dg$
  and $u_{\vars,\dg}$ is convex, it is sufficient to show that
  \begin{align*}
    \vars+\dg = u_{\vars,\dg}(\dg) < u_{\vars,\dg}(1) =
    \binom{\vars+\dg} {\vars}.
  \end{align*}
  The inequality is equivalent to
  $$
  \vars! < (\vars+\dg-1)^{\underline {\vars -1}},
  $$
  where $a^{\underline \vars} = a \cdot (a-1) \cdots (a-\vars+1)$ is
  the falling factorial (or Pochhammer symbol). This is valid for
  $\dg=2$ since $2 < \vars +1$, and the right hand side is monotonically
  increasing in $\dg$, so that the claim is proven.

  It follows that $D_{\vars,\dg}$ is contained in a proper closed
  subset of $P_{\vars,\dg}^{=}$, and there is a dense open subset
  consisting of indecomposable polynomials, which is
  \short\ref{thm:subvar-2}. This fact also holds in each
  $P_{\vars,\dg/e}^{=}$, and in $P_{\vars,\dg/e}^{0}$ by
  \ref{invariant}.  From the uniqueness of normal decompositions with
  indecomposable right factor (\ref{fact:uni}) we conclude that each
  fiber of the restriction of $\gamma_{\vars,\dg,e}$ to $P_{1,e}^{=}
  \times I_{r, \dg/e}^{0}$ consists of a single point.  Thus equality holds in
  \ref{dim}, and \short\ref{thm:subvar-1} is also proven.

  \short\ref{thm:subvar-3} For superlinear compositions, we have $
  D_{\vars,\dg}^{~sl} = \varnothing$ if $\dg$ is prime, and now may
  assume $\dg$ to be composite. The maximal value allowed for $e$ in
  \ref{substack} is $\dg/l$.  Thus \short\ref{thm:subvar-3} follows
  from \short\ref{thm:subvar-1} when $m<\dg$.  For $\vars=2$,
  \begin{equation}\label{suli}
    u_{2,\dg}(l)-u_{2,\dg}(\dg/l)=\frac{(\dg-l^{2})(\dg+l^{2}+l)}{2l^{2}}
  \end{equation}
  is always nonnegative, so that
  \begin{align*}
    \dim D^{~sl}_{2,\dg}= \dim ~im \gamma_{2,\dg,l}=u_{2,\dg}(l).
  \end{align*}
  Together with the uniqueness of \ref{fact:uni}, this proves
  \short\ref{thm:subvar-3} also for $\vars=2$.
\end{proof}

\section{Counting decomposables over finite fields}\label{sec:cdoff}

The goal in this section is to approximate the number of multivariate
decomposables over a finite field, with a good relative error bound.

Over a finite field $F = \mathbb{F}_{q}$ with $q$ elements, we have
\begin{align*}
  \# P_{\vars,\dg}^{=} & = q^{b_{\vars,\dg}} -q^{b_{\vars,\dg-1}} =
  q^{b_{\vars,\dg}}
  (1-q^{-b_{\vars-1,\dg}}),\\
  \#P_{\vars,\dg}^{0}& = \frac{\#P_{\vars,\dg}^{=}}{q \cdot (q-1)} =
  q^{b_{\vars,\dg}-2}\frac{1-q^{-b_{\vars-1,\dg}}}{1-q^{-1}}.
\end{align*}

The proof of the following estimate of $\#D_{\vars,\dg}$ involves
several case distinctions which are reflected in the somewhat
complicated statement of the theorem. A simplified version is
presented in \ref{cor:triv} below. 

\begin{theorem}
  \label{thm:Count}
  Let $F = \F_{q}$ be a finite field with $q$ elements, $r\geq 2$, $l$
  the smallest prime divisor of $n \geq 2$, and $m$ as in \ref{defM}. We set
  \begin{align}\label{eq:pm}
    \displaystyle
    \alpha_{\vars,\dg} ~ & ~ =
    q^{\binom{\vars+\dg/m}{\vars}+m-1}(1-q^{-\binom{\vars-1+\dg/m}{\vars-1}}),
    \\
    c_{\vars,\dg,1} ~&~ =  l-3,\nonumber \\
    \displaystyle
    c_{\vars,\dg,2}~&~ =  l-2, \nonumber\\
    c_{\vars,\dg,3}~&~ = \binom{\vars+1}{2}-2, \nonumber\\
    c_{\vars,\dg,4} ~&~= \binom{\vars-1+\dg/l}{\vars -1}-1,\nonumber
  \end{align} 
 
  \begin{equation}
    \label{defBeta}
    \beta_{\vars,\dg} = 
    \begin{cases}
      0 & \textrm{if }  \dg \textrm{ is prime}, \\
      \displaystyle \frac {2q^{-c_{\vars,\dg,1}} (1-q^{-\dg/l-1})}
      {1-q^{-2}} &
      \textrm{if \ref{special} holds}, \\
      2q^{-c_{\vars,\dg,2}} &
      \text{if $\vars=2$ and $\dg/l=2l-3$ is prime},  \\
      q^{-c_{\vars,\dg,3}} & \text{if $\dg =4$},\\
      \displaystyle \frac{2q^{-c_{\vars,\dg,4}}} {1-q^{-1}} &
      \textrm{otherwise}.
    \end{cases}
  \end{equation}
  Then the following hold.
  \begin{enumerate}
  \item
    \label{thm:Count-1}
    $$
    \left|{\#D_{\vars,\dg}}- \alpha_{\vars,\dg}\right| \leq
    \alpha_{\vars,\dg}\cdot \beta_{\vars,\dg}.
    $$
  \item
    \label{thm:Count-2}
    $$
    {\#I_{\vars,\dg}} \geq \#P^{=}_{\vars,\dg} -2 \alpha_{\vars,\dg}.
    $$
  \item We set
    \label{thm:Count-3}
    \begin{align*}
      \alpha_{\vars,\dg}^{~sl} &=
      \begin{cases}
        0 & \text{if $\dg$ is prime},\\
        q^{\binom{2+\dg/l}{2}+l-1}(1-q^{-\dg/l-1})
        & \text{if \ref{special} holds}, \\
        \alpha_{\vars,\dg} & \text{otherwise},
      \end{cases}
      \\
      \beta_{\vars,\dg}^{~sl} &=
      \begin{cases}
        q^{-(\dg+l^{2}+l)(\dg-l^{2})/2l^{2}} &
        \textrm{if \ref{special} holds and } n > l^{2},\\
        q^{-(\dg+l-2)/2} & \text{if \ref{special} holds and } \dg = l^{2},\\
        \beta_{\vars,\dg} & \text{otherwise}.
      \end{cases}
    \end{align*}
    Then
    \begin{equation}
      \label{thm:Count-3.1}
      \left|{\#D_{\vars,\dg}^{~sl}}- \alpha_{\vars,\dg}^{~sl}\right|
      \leq \alpha_{\vars,\dg}^{~sl} \cdot \beta_{\vars,\dg}^{~sl}.
    \end{equation}
  \item
    \label{thm:Count-4}
    $\#I^{~sl}_{\vars,\dg} \geq \# P^{=}_{\vars,\dg}
    -2\alpha^{~sl}_{\vars,\dg}$.
  \end{enumerate}
\end{theorem}

\begin{proof}
  The proof of \short\ref{thm:Count-1} and \short\ref{thm:Count-2}
  proceeds in three stages: an upper bound on decomposables, a lower
  bound on indecomposables, and a lower bound on decomposables. Each
  stage depends on the previous one.

  According to \ref{defBeta}, we have to distinguish five cases:
  $$
  \begin{array}{c|l|c|c}
    i & \text{condition for case $i$} & m & c_{\vars,\dg,i} \\\hline
    0 & \text{$\dg$ prime} & \dg & \\\hline
    1 & \text{$\vars=2$, $\dg/l \leq 2l-5$ prime} & \dg & l-3 \\\hline
    2 & \text{$\vars=2$, $\dg/l=2l-3$ prime} & l & l-2 \\\hline
    3 & \text{$\dg=4$} & l & \displaystyle \binom{\vars+1}{2}-2\\\hline
    4 & \text{otherwise} & l & \displaystyle \binom{\vars-1+\dg/l}{\vars-1}-1 \\
  \end{array}
  $$

  In the first stage, for a divisor $e$ of $\dg$, we have
  \begin{align*}
    \# ~im \gamma_{\vars,\dg,e} & \leq \# P^{=}_{1,e} \cdot
    \#P^{0}_{\vars, \dg/e} = q^{b_{\vars,\dg/e}+e-1}
    (1-q^{-b_{\vars-1,\dg/e}}),
  \end{align*}
  and thus with $u_{\vars,\dg}$ from \ref{real}

  \begin{equation}\label{mon}
    \# D_{\vars,\dg} \leq \sum_{1< e\mid\dg}\# ~im \gamma_{\vars,\dg,e}
    \leq 
    \sum_{1< e\mid\dg}q^{u_{\vars,\dg}(e)}
    (1-q^{-b_{\vars-1,\dg/e}}).
  \end{equation}

  We write $u$ for $u_{\vars,\dg}$ and $c_{i}$ for $c_{\vars,\dg,i}$,
  and recall $E=\{e \in \mathbb{N}\colon 1 < e\mid\dg \}$.

  If $\dg$ is prime, then $E=\{\dg\}$, $m=l=\dg$ (see \ref{defM}), and
  each right hand component $h$ in a decomposition is linear, hence
  indecomposable. It follows from \ref{fact:uni} that
  $\gamma_{\vars,\dg,\dg}$ is injective, $D_{\vars,\dg}= ~im
  \gamma_{\vars,\dg,\dg}$, and $\# D_{\vars,\dg} =
  \alpha_{\vars,\dg}$. All claims follow in this case.

  In the first stage, we may use the following blanket assumptions and
  notations:
  \begin{equation}
    \label{blanket}
    \vars\geq 2, a=\dg/l \geq \sqrt \dg \geq l \geq 2, a^{2}\geq \dg\geq 2l \geq l+2.
  \end{equation}


  We first explain our general strategy for the upper bound
  \begin{equation}\label{strat}
    \# D_{\vars,\dg} \leq \alpha_{\vars,\dg} (1+\beta_{\vars,\dg})
  \end{equation}
  in \short\ref{thm:Count-1}. From \ref{maxU} we know that the maximal
  value of $u$ occurs at $e=m$.  By the convexity of $u$, each value
  is assumed at most twice, and we can majorize the sum in \ref{mon}
  by twice a geometric sum.  However, this would provide an
  unsatisfactory error estimate, and we want to show that the
  difference between $u(m)$ and the other values $u(e)$ with $e\in E$
  is sufficently large. We abbreviate
  $$
  w = \frac {1-q^{-b_{\vars-1,\dg/l}}} {1-q^{-b_{\vars-1,\dg/m}}} ,
  $$
  define $\delta$, $\mu$, and $\beta$ in \ref{method}, and claim that
  for any $c$ the following implication holds:
  \begin{equation}
    \left.\begin{array}{rcl}\label{method}
        c \leq \delta & = & \min_{e \in E\smallsetminus \{m\}} (u(m)-u(e))\\
        \mu & = & \min\{\#E -1, \frac{2}{1-q^{-1}}\} \\
        \beta & = & \mu w q^{-c}\\
      \end{array}\right\} \Rightarrow \# D_{\vars,\dg} \leq \alpha_{\vars,\dg}(1+\beta).
  \end{equation}
  In our four cases, $c$ will be instantiated by $c_{1}$, $c_{2}$, $c_{3}$,
  and $c_{4}$.  We note that $\mu\leq 4$. In order to prove the claim,
  we note that
  $$
  u(e)-u(m)\leq - c
  $$
  for all $e \in E \smallsetminus \{m\}$.  Since $b_{\vars-1,k}$ is
  monotonically increasing in $k$ and $\dg/e \leq \dg/l$, we have
  \begin{align*}
    1-q^{-b_{\vars-1,\dg/e}} \leq 1-q^{-b_{\vars-1, \dg/l}}
  \end{align*}
  for all $e \in E$. Using this estimate for all $e \neq m$ and the
  fact that the convex function $u$ takes any of its values at most twice, we
  find that
  \begin{align*}
    q^{-u(m)} \sum_{e\in E}q^{u(e)} {(1-q^{-b_{\vars-1,\dg/e}})} & <
    (1+2w \sum_{i \leq -c}q^{i}) \cdot {(1-q^{-b_{\vars-1,\dg/m}})}
    \\
    & = ( 1 + \frac{2w q^{-c}}{1-q^{-1}} ) \cdot
    {(1-q^{-b_{\vars-1,\dg/m}})}.
  \end{align*}
  Also, since $E\smallsetminus \{m\}$ has $\#E -1$ elements, we find
  \begin{align*}
    q^{-u(m)} \sum_{e\in E}q^{u(e)} {(1-q^{-b_{\vars-1,\dg/e}})} \leq
    ( 1 + {(\#E -1)w q^{-c}}) \cdot {(1-q^{-b_{\vars-1,\dg/m}})}.
  \end{align*}
  Using \ref{mon} we conclude that
  \begin{align}
    \#D_{\vars,\dg} \leq q^{u(m)}(1-q^{-b_{\vars-1,\dg/m}})\cdot (1+
    {\mu w q^{-c}}) = \alpha_{\vars,\dg}(1+\beta),
  \end{align}
  as claimed. It then remains to see that $\beta \leq
  \beta_{\vars,\dg}.$

  We now turn to our four cases. In case 1, \ref{special} holds, $E =
  \{l, \dg/l, \dg\}$, $r=2$, $l\geq 5$, $m=\dg$, and
  $$
  w = \frac{1-q^{-\dg/l-1}}{1-q^{-2}} .
  $$
  Now \ref{suli} says that
  $$
  u(l)-u(\dg/l) = \frac {(\dg-l^{2})(\dg+l^{2}+l)}{2l^{2}} \geq 0,
  $$
  so that $u(e) \leq u(l)$ for all $e \in E\smallsetminus \{m\}=\{l,\dg/l\}$, and
  by  \ref{ul-ud}
  \begin{equation*}
    \label{dlDiff}
    \delta =
    u(\dg)-u(l) = \frac 1 2 (\frac \dg l -1)(2l-4-\frac \dg l ) > 0. 
  \end{equation*}
  The two right hand factors are positive integers. If the second one
  equals 1, then
  $$
  \delta = \frac 1 2 (2l-5-1) = l-3 = c_{1}.
  $$
  Otherwise, $\delta \geq \dg/l-1 \geq l-1 > l-3=c_{1}$. Thus the
  assumptions in \ref{method} hold with $c=c_{1}$, and since $\# E
  \leq 3$, we have $\mu \leq 2$ and $\beta \leq 2wq^{-c} =
  \beta_{\vars,\dg}$. This shows \ref{strat} in case 1.

  In case 2, we have $E=\{l,2l-3,\dg\}, m=l$, and
  \begin{align*}
    &u(l)-u(\dg)= l-2,\\
    &u(l)-u(2l-3)=\frac{1}{2}(l-3)(3l-2).
  \end{align*}
  The minimum of these two values is $l-2$ when $l \geq 5$. Then $\delta =
  l-2=c_{2}$, and furthermore $\mu=2$ and $w=1$. This implies
  \ref{strat} in case 2, when $l \geq 5$. For $l=3$, we have $\dg =
  9$, $E=\{3,9\}$, $u(3)=12$, $u(9)=11$, $\delta = 1 = l-2 =
  c_{2}$, $\mu = 1$, and $w=1$. Thus $\beta=q^{-c_{2}}<\beta
  _{r,\dg}$, and \ref{strat} again holds.

  In case 3, we have $E=\{2,4\}$, $l=m=2$, $w=\mu=1$,
  $$
  \delta= u(2)-u(4)=\binom{\vars+1}{2}-2 = c_{3} \geq 1,
  $$
  and \ref{strat} holds.

  In case 4, we have $m=l < \dg$, and introduce $l^{*}=\dg
  l/(\dg-l)\in \mathbb{Q}$. ($l^{*}$ is not an integer unless $\dg$
  is 4 or 6.)  We first claim that
  \begin{align}
    \label{l0vsd}
    u(\dg) \leq u(l^{*}).
  \end{align}

  We start with the subcase $\vars\geq 3 $ and have to show that
  \begin{equation}
    \label{case2n2}
    \binom{\vars+a-1} \vars + \frac \dg {a-1} - 1 
    = u(l^{*}) \geq u(\dg) = \vars+\dg.
  \end{equation}
  We first treat the subcase $a\geq 5$. Then $a^{3} \geq 3a^{2} +
  4a +12$, so that the first inequality in
 \begin{equation} 
 \begin{aligned}
   \label{case2n2new}
   \frac 1 {a-1} \binom{\vars+a-2} {a-2} = \frac 1 {\vars+a-1} \binom{\vars+a-1}{\vars} \\
   \geq 1 + \frac {a^{2}} {\vars+a-1} \geq 1 + \frac \dg {\vars+a-1}
  \end{aligned}
\end{equation}
  is valid for $\vars=3$, and for all $\vars\geq 3$ since the left
  hand side is monotonically increasing and the right hand side
  decreasing in $\vars$. Using \ref{blanket}, this yields
  \ref{case2n2}.

  In the remaining subcase $\vars \geq 3$ and $a\leq 4$, we have $\dg
  \in \{4,6,8,9\}$. Case 3 covers $\dg=4$.  The inequality between the
  outer terms in \ref{case2n2new} holds for the following values of
  $(\vars,\dg)$: $(4,6)$, $(3,8)$, and $(4,9)$, and by monotonicity
  for these values of $\dg$ and any larger $\vars$. One checks \ref{case2n2}
  for $(3,6)$ and $(3,9)$.

  We next have the subcase $\vars=2$ and $a\geq 3$. Then
  \begin{align}\label{eq:subcase}
  u(\dg) - u(l^{*})& = \frac {a-2} {2a-2} \cdot (2\dg-a^{2}-2a+3),\\\nonumber
  u(\dg)>u(l^{*})&\Longleftrightarrow  2 a l = 2\dg > a^{2} +2a-3 \\\nonumber
&\Longleftrightarrow 2l > a + 2 - \frac
    3 a \Longleftrightarrow 2l \geq a + 2 \Longleftrightarrow 2 l-2
    \geq a .\nonumber
  \end{align}
  By assumption, \ref{special} does not hold, and if \ref{eq:subcase}
  is positive, then $2l-4 \leq a \leq
  2l-2$ follows. If $a$ is even, then $l =2$, and one finds that
  $\dg=4$, which is case 3. So the only remaining possibility is
  $a=2l-3$. Since each prime divisor of $a$ is at least $l$, $a$ is
  prime. But this is case 2, and therefore  \ref{l0vsd} holds.

  For the remaining possibility $a=2$, we find $l=2$ and $\dg=4$,
  which has been dealt with.
  We conclude that \ref{l0vsd} always holds in
  case 4.

  We have
  $$
  l^{2}+2l < 2\dg
  $$
  for all $\dg \neq 4$, since this follows from $\dg \geq l^{2}$ when $
  l\geq 3$, and also for $l=2$. This implies that
  $$
  l^{*}-l = \frac{l}{\dg/l-1}<2.
  $$
  For any $e \in E \smallsetminus \{l\}$, we have $l < e \leq \dg$ and
  $\dg/e < \dg/l$. These values are both integers, so that
  $$
  \frac{\dg}{e} \leq \frac{\dg}{l}-1= \frac{\dg}{l^{*}}.
  $$
  Thus $l^{*} \leq e \leq \dg$ for all $e \in E \smallsetminus
  \{l\}$. By \ref{l0vsd} and the convexity of $u$, the maximal value of
  $u(e)$ for these $e$ is at most $\max \{u(l^{*}), u(\dg)\}
  =u(l^{*})$. In \ref{method} we have

  \begin{align*}
    \delta & \geq
    u(l)-u(l^{*})=\binom{\vars+\dg/l}{\vars}-\binom{\vars-1+\dg/l}{\vars}+
l-l^{*} \\
    & = \binom{\vars-1+\dg/l}{\vars-1}+l -l^{*}>c_{4}+1-2=c_{4}-1.
  \end{align*}
  Since $\delta$ and $c_{4}$ are integers, we also have $ \delta \geq
  c_{4}$. Furthermore, we have $w=1$ and $\mu \leq 2(1-q^{-1})^{-1}$,
  so that $\beta \leq \beta_{\vars, \dg}$. Then the assumptions in
  \ref{method} hold with $c=c_{4}$, and \ref{strat} follows.

  In the next stage, we derive the lower bound in
  \short\ref{thm:Count-2} on the number $\# I_{\vars,\dg}$ of
  indecomposable polynomials.  The previous results yield
  \begin{align*}
    \# P_{\vars,\dg}^{=} - \# I_{\vars,\dg} = \#D_{\vars,\dg} \leq
    \alpha_{\vars,\dg} (1+\beta_{\vars,\dg}).
  \end{align*}
  The claim in \short\ref{thm:Count-2} is that the last expression is
  at most $2 \alpha_{\vars,\dg}$, that is, $\beta_{\vars,\dg}\leq 1$.
  Again, we distinguish according to our four cases.

  For case 1, we have $l \geq 5$ and $(1-q^{-2})^{-1} \leq 4/3$, and
  thus $\beta_{\vars,\dg} < \frac 8 3 q^{-l+3} \leq \frac 8 3 \cdot
  2^{-2} < 1$.

  In case 2, we have $l \geq 3$ and
  $$
  \beta_{\vars,\dg} = 2q^{-l+2} \leq q^{-l+3} \leq 1.
  $$

  In case 3, we have $c_{3}=\binom{\vars+1}{2}-2 \geq 1 >0$ and
  $\beta_{\vars,4}=q^{-c_{3}} < 1.$

  In case 4,
  we have $\beta_{\vars,\dg} \leq 4q^{-c_{4}} \leq q^{2-c_{4}}$, so
  that it is sufficient to show that $ c_{4} \geq 2$. 
  We have $\vars, a \geq 2$ and
  $$
  c_{4} + 1= \binom{\vars-1+a} { \vars-1} \geq \binom{\vars+1}
  {\vars-1} = \frac {\vars\cdot (\vars+1)} 2 \geq 3.
  $$
   This concludes the proof of \short\ref{thm:Count-2}.

  In the last stage, we estimate the number of decomposable
  polynomials from below. The idea is obvious: we take the largest
  type of decomposable polynomials, as identified above, and then use
  only indecomposable polynomials as right components, so that the
  uniqueness property of \ref{fact:uni} applies.  We set
  $$
  I_{\vars,\dg}^{0} = I_{\vars,\dg} \cap P_{\vars,\dg}^{0}.
  $$
  By \ref{invariant} and \short\ref{thm:Count-2}, we have
  \begin{align*}
    \# I_{\vars,\dg}^{0}& = \# I_{\vars,\dg} \cdot \frac {\#
      P_{\vars,\dg}^{0}} {\# P_{\vars,\dg}^{=}} \geq \frac {(\#
      P_{\vars,\dg}^{=} -2\alpha_{\vars,\dg}) \cdot {\#
        P_{\vars,\dg}^{0}} }
    {\# P_{\vars,\dg}^{=}}\\
    & = (1-\frac {2\alpha_{\vars,\dg}} {\# P_{\vars,\dg}^{=}} ) \frac
    {q^{b_{\vars,\dg}-2} (1-q^{-b_{\vars-1,\dg}})} {1-q^{-1}}.
  \end{align*}
  Thus
  \begin{align*}
    \# D_{\vars,\dg} & \geq \# \gamma_{\vars,\dg,m} (P_{1,m}^{=}
    \times I_{\vars,\dg/m}^{0} )
    = \# (P_{1,m}^{=} \times I_{\vars,\dg/m}^{0} )\\
    & \geq q^{b_{\vars,\dg/m}+m-1} (1-\frac {2\alpha_{\vars,\dg/m}} {\#
      P_{\vars,\dg/m}^{=}}) (1-q^{b_{\vars-1,\dg/m}}) =
    \alpha_{\vars,\dg} \cdot (1-\frac {2\alpha_{\vars,\dg/m}} {\#
      P_{\vars,\dg/m}^{=}}).
  \end{align*}

  In the cases 2 and 3, $\dg/m$ is prime, $\beta_{r, \dg/m}=0$, and we
  could replace the factor $2$ in the last expression by $1$; however,
  we do not need this in the following. In order to prove the lower
  bound in \short\ref{thm:Count-1}, we proceed according to our four
  cases. In case 1, we have $\vars=2$, \ref{special}\ holds, $m=\dg$,
  and
  \begin{equation}\label{eq:lb}
  \#D_{\vars,\dg} \geq \# ~im \gamma_{\vars,\dg,\dg} = \#
  (P_{1,\dg}^{=} \times P_{\vars,1}^{0} ) = \alpha_{\vars,\dg}.
  \end{equation}

  For the remaining three cases, we have $m=l$ and claim that
  \begin{equation}
    \label{alphabeta}
    \frac {2 \alpha_{\vars,\dg/l}} {\# P_{\vars,\dg/l}^{=}} \leq \beta_{\vars,\dg},
  \end{equation}
  from which the lower bound follows:
  \begin{align*}
    \# D_{\vars,\dg} \geq \alpha_{\vars,\dg} \cdot (1- \frac {2
      \alpha_{\vars,\dg/l}} {\# P_{\vars,\dg/l}^{=}} ) \geq
    \alpha_{\vars,\dg} \cdot (1-\beta_{\vars,\dg}).
  \end{align*}

  We denote by $m^{*}$ the quantity defined in \ref{defM} for the
  argument $a = \dg/l$ instead of $\dg$ (and hence using the smallest prime
  divisor of $\dg/l$ instead of $l$), and set
  $d=a/m^{*}=\dg/lm^{*}$. Thus $m^{*}$ is either $a$ or its smallest
  prime divisor, $a =m^{*} d\geq 2d \geq 2$, and
  \begin{equation}
    \label{alpha*}
    \frac {2 \alpha_{\vars,a}} {\# P_{\vars,a}^{=}} 
    = \frac { 2q^{-c^{*}}
      (1-q^{-b_{\vars-1,d}}) } 
    {1-q^{-b_{\vars-1,a}}} \leq 2q^{-c^{*}},
  \end{equation}
  with
  \begin{equation*}
    \label{c*}
    c^{*} =
    \binom{\vars+a} {\vars}-\binom{\vars+d} {\vars}-m^{*}+1. 
  \end{equation*}

It is therefore sufficient for \ref{alphabeta} to show 
\begin{equation} \label{eq:2q}
2q^{-c^{*}} \leq \beta_{\vars, \dg}.
\end{equation}

In case 2, $m^{*}= a =\dg/l = 2l-3$ is prime, and
  \begin{align*}
    c^{*} & = (2l-1)(l-2) > l-2,\\
    2q^{-c^{*}} & < 2q^{-(l-2)}=\beta_{2,\dg},
  \end{align*}
  and \ref{eq:2q} is satisfied.

In case 3, we have $n=4$, $l=2$, $a=m^{*}=2$, $d=1$, $c^{*}=
\binom{r+1}{2}-1$, and thus
$$
2q^{-c^{*}} \leq q \cdot q^{-\binom{r+1}{2}+1}= \beta_{\vars,4}.
$$

  In case 4, we have
  $$
  \beta_{\vars,\dg} = \frac {2q^{-c_{4}}} {1-q^{-1}} > 2q^{-c_{4}},
  $$
  and it is sufficient for \ref{eq:2q} to show that
  \begin{equation}
    \label{c*vsci}
    c^{*} \geq c_{4},
  \end{equation}
  which in turn amounts to showing that
  \begin{align}
\begin{aligned}
    \label{c*vsc3}
    \binom{r-1+a}{r} = \binom{\vars+a} {\vars} 
    &- \binom{\vars-1+a} {\vars-1} 
    \geq
    \binom{\vars+d} {\vars} 
    + m^{*}-2,
\end{aligned}
\end{align}  
using Pascal's identity. We prove this by induction on $r \geq 2$. For
$r=2$, we use $a=m^{*} d \geq m^{*} \geq 2$. Thus
$$
a^{2}+a-(\frac{a}{m^{*}})^{2}-3\frac{a}{m^{*}}=\frac{a}{(m^{*})^{2}}\bigl(a((m^{*})^{2}-1)+(m^{*})^{2}-3m^{*}\bigr) \geq 2m^{*}-2,
$$
since the inequality holds for $a=m^{*}$ and the middle term is monotonically
increasing in $a$ for $m^{*} \geq 2$. It follows that
$$
a^{2}+a \geq (\frac{a}{m^{*}})^{2}+3 \frac{a}{m^{*}}+2m^{*}-2,
$$
which implies \ref{c*vsc3} for $r=2$.

For the induction step, we have $a-1 \geq a/2 \geq a/m^{*}=d$, and 
$$
\binom{r+a-1}{r}- \binom{r+d}{r} \geq
\binom{r-1+a-1}{r-1}-\binom{r-1+d}{r-1}\geq m^{*}-2,
$$
again by Pascal.

  This finishes the proof of \short\ref{thm:Count-1}, and it remains
  to prove \short\ref{thm:Count-3} and \short\ref{thm:Count-4}. We may
  assume $\dg$ to be composite. Since
  $D^{~sl}_{\vars,\dg} \subseteq D_{\vars,\dg}= D^{~sl}_{\vars,\dg}
  \cup ~im \gamma_{\vars,\dg,\dg}$, the upper bound on $\#D_{\vars,\dg}$ in
  \short\ref{thm:Count-1} also holds for $\#D^{~sl}_{\vars,\dg}$, and
  the lower bound does unless $m=\dg$. Thus \short\ref{thm:Count-3}
  and \short\ref{thm:Count-4} follow unless \ref{special} holds, which
  we now assume.

  Since $\dg/l \geq l$, we have $1-q^{-\dg/l-1} \geq
  1-q^{-l-1}$. Using \ref{suli}, we find
  \begin{align*}
    \#D^{~sl}_{2, \dg} & \leq \#(P^{=}_{1,l} \times P^{0}_{2,\dg/l})
    + \# (P^{=}_{1,\dg/l} \times P^{0}_{2,l})\\
    & = \alpha^{~sl}_{2,\dg}(1+q^{-(n+l^{2}+l)(n-l^{2})/2l^{2}}
    \frac{1-q^{-l-1}}
    {1-q^{-\dg/l-1}}) \leq \alpha^{~sl}_{2,\dg}(1+ \beta^{~sl}_{2,\dg}),\\
    \# D^{~sl}_{2,\dg} & \geq \#(P^{=}_{1,l} \times I^{0}_{2,\dg/l})\\
    & \geq \#P_{1,l}^{=} \cdot (\#P_{2,\dg/l}^{=} - 2 \alpha_{2,\dg/l})
    \cdot
    \frac{\#P_{2,\dg/l}^{0}}{\#P_{2,\dg/l}^{=}}\\
    & = \alpha_{2,\dg}^{~sl}(1-2q^{-(\dg+2l)(\dg-l)/2l^{2}}
    \frac{1-q^{-2}}
    {1-q^{-\dg/l-1}})\\
    & \geq \alpha_{2,\dg}^{~sl}(1-q^{-(\dg+2l)(\dg-l)/2l^{2}+1})\\
    & > \alpha^{~sl}_{2,\dg}(1- \beta^{~sl}_{2,\dg}).
  \end{align*}

 If $\dg=l^{2}$, then $D^{~sl}_{2,\dg}= ~im \gamma_{2,\dg,l}$ and
  \begin{align*}
    \# D^{~sl}_{2,\dg} & \leq \# (P^{=}_{1,l} \times P^{0}_{2,l}) =
    \alpha^{~sl}_{2,\dg},\\
    \#D^{~sl}_{2,\dg} & \geq \#(P^{=}_{1,l} \times I^{0}_{2,l}) \geq
    \alpha^{~sl}_{2,\dg}(1-\beta^{~sl}_{2,\dg}
    \frac{1-q^{-2}}{1-q^{-l-1}}) \geq
    \alpha^{~sl}_{2,\dg}(1-\beta^{~sl}_{2,\dg}).\qed
  \end{align*}
\end{proof}

\begin{remark}
  In the simple case where $\dg$ has exactly two prime factors and $r
  \geq 2$, it is easy to determine $\#D_{\vars,\dg}$ exactly.  For $\dg =
  l^{2}$,
$$
D_{\vars,\dg}= \gamma_{\vars, \dg,l}(P^{=}_{1,l} \times I^{0}_{\vars,l})
\cup \gamma_{\vars, \dg, \dg}(P^{=}_{1, \dg} \times I^{0}_{r,1})
$$ 
is a disjoint union. We have
\begin{align*}
\#D_{\vars,\dg} =
\begin{cases}
 \alpha_{\dg}+q^{(\dg+5l)/2}(1-q^{-l-1})-q^{2l+1}(1-q^{-r})& \text{ if
  \ref{special} holds},\\
 \alpha_{\dg}+q^{\dg+\vars}(1-q^{-\vars})(1-q^{2l-\dg-1})& \text{ otherwise}.
\end{cases}
\end{align*}
We set
\begin{align*}
\beta'_{\vars, \dg}=
\begin{cases}
  \displaystyle q^{(-\dg+5l-4)/2}
  \frac{1-q^{-l-1}}{1-q^{-2}} -q^{-\dg+2l-1} & \text{if }
  \ref{special} \text{ holds},\\
  \displaystyle q^{\dg+\vars+1-\binom{\vars+\dg/l}{\vars}-l}
  \frac{(1-q^{-\vars})(1-q^{2l-\dg-1})}{1-q^{-\binom{\vars-1+\dg/l}{\vars-1}}}&
  \text{otherwise}.
\end{cases}
\end{align*}
Then
$$
\#D_{\vars, \dg}= \alpha_{\vars, \dg}(1+ \beta'_{\vars, \dg}).
$$

This value is exact, in contrast to the estimates of
\ref{thm:Count}, and $\beta^{'}_{\vars, \dg}$ is often much smaller
than $\beta_{\vars, \dg}$. The drawback is that the values are more
complicated, and an attempt to generalize this approach to more
than two prime factors of $\dg$ does not seem to lead to manageable results.

If $n > l^{2}$ and $\dg/l$ is prime, then one finds similarly that

\begin{align*}
  \#D_{\vars,\dg} & =
  q^{b_{\vars,\dg/l}+l-1}(1-q^{-b_{\vars-1,\dg/l}})+q^{b_{\vars,l}+\dg/l-1}(1-q^{b_{\vars-1,l}})\\
  & \quad + q^{\dg+\vars}(1-q^{-\vars})(1-2q^{l+\dg/l-\dg-1}).
\end{align*}

Here it is not even transparent which of the summands is the dominating
one. However, using the case distinction of \ref{special}, one again
obtains a quantity $\beta^{'}_{\vars,\dg}$, so that $\#D_{\vars,\dg}=
\alpha_{\vars,\dg}(1+\beta^{'}_{\vars,\dg})$. The previous remarks
apply to this solution as well. 
\end{remark}

\cite{boddeb09} obtain an equivalent result, in a different language.
They also show that $\#I_{\vars,\dg}/ \#P_{\vars,\dg}^{=} \rightarrow
1$ as $\dg \rightarrow \infty$ 
(see \ref{thm:Count-2}), and some results similar to those of
\ref{thm:Count-1} when either $r=2$ or $\dg$ has at most two
prime factors. Their methods do not lead to
a unified formula as in \ref{thm:Count-1}, and the error
bounds are weaker than the present ones by factors of $O(\dg)$ or $O(q)$.

If $u_{2,\dg}(e)=u_{2,\dg}(e')$ never happened for distinct divisors
$e$, $e'\geq 2$ of $\dg$, we could save a factor of $2$ in
$\beta_{2,\dg}$. However, if we take two arbitrary positive integers
$k \geq 2$ and $m$, set $e = 2km^{2}+2m^{2}+3m$, $e'=ke$, and
$\dg=2mke$, then $e < e'$ and $u_{2, \dg}(e)= u_{2,\dg}(e')$. The
smallest such choice gives $\dg=36$, $e=9$, $e'=18$.

We can unify cases 2 and 4 in \ref{defBeta}, and the other cases fit
in trivially. We set
\begin{equation}
  \label{fit}
  \begin{aligned}
    c_{\vars,\dg,5}& = \frac 1 2 \binom{\vars-1+\dg/l}{\vars-1}
    -1,\\
    \beta_{\vars,\dg}^{*} & = \frac {2q^{-c_{\vars,\dg,5}}}
    {1-q^{-1}}.
  \end{aligned}
\end{equation}

\begin{corollary}\label{cor:triv}
  Let $D_{\vars,\dg}$ be the set of decomposable polynomials of degree
  $n\geq 2$ in $r \geq 2$ variables over $\mathbb{F}_{q}$, and
  $\alpha_{\vars, \dg}$ and $\beta_{\vars, \dg}^{*}$ as in \ref{eq:pm}
 and
  \ref{fit}, respectively. Then
  $$
  \left| \# D_{\vars,\dg} - \alpha_{\vars,\dg} \right| \leq
  \alpha_{\vars,\dg} \cdot \beta_{\vars,\dg}^{*}.
  $$
\end{corollary}
\begin{proof}
  It is sufficient to show that $ \beta_{\vars,\dg} \leq
  \beta^{*}_{\vars,\dg}$ in all cases. This is an easy calculation.
\end{proof}

How close is our relative error estimate $\beta_{\vars,\dg}$ to being
exponentially decaying in the input size? In the ``general'' case 4 of
\ref{defBeta}, $\beta_{\vars,\dg}$ is about $q^{-c_{4}}$ with $c_{4}$
approximately $b_{\vars-1, \dg/l}=
\binom{\vars-1+\dg/l}{\vars-1}$. \ref{fit} and \ref{cor:triv} relate
also the special cases to this.

The (usual) dense representation of a polynomial in $\vars$ variables
and of degree at most $\dg$ requires
$b_{\vars,\dg}=\binom{\vars+\dg}{\vars}$ monomials, each of them
equipped with a coefficient from $\mathbb{F}_{q}$, using about
$\log_{2}q$ bits. Thus the total input size is about $\log_{2}q \cdot
b_{\vars,\dg}$ bits. Now $\log_{2}q \cdot b_{\vars,\dg/l}$ differs from
$\log_{2}\beta_{\vars,\dg}$ by a factor of $1+ \frac{\dg}{\vars l}$.
Furthermore, $\dg$ and $\dg/l$ are polynomially related, since $\dg >\dg/l
\geq \sqrt{\dg}$. Up to these polynomial
differences (in the exponent), $\beta_{\vars, \dg}$ is exponentially decaying in the
input size. Furthermore $\beta_{\vars,\dg}$ is exponentially decaying
in any of the parameters $\vars$, $\dg$ and $\log_{2}q$, when the
other two are fixed.

We compare our results to those of \cite{gat08-incl-gat07} on the number
$\#R_{\dg}$ of reducible and $\#E_{\dg}$ of relatively irreducible
(irreducible and not absolutely irreducible) bivariate
polynomials. Ignoring small factors and special cases like
\ref{special}, we have for composite $\dg$
\begin{align*}
\#R_{\dg} & \approx q^{\binom{\dg+2}{2}-\dg+1}\\
\#E_{\dg} & \approx q^{\binom{\dg+2}{2}- \frac{\dg^{2}(l-1)}{2l}}\\
\#D_{2,\dg} & \approx q^{\binom{\dg/l+2}{2}+l-1}.
\end{align*} 

The first exponent is always greater than the third one, and for the
second and third ones we have
$$
\binom{n+2}{2}- \frac{\dg^{2}(l-1)}{2l}-\binom{\dg/l+2}{2}-l+1
= \frac{(l-1)(\dg^{2}+3\dg l-2l^{2})}{2l^{2}}>0.
$$
In other words, there are many more reducible or relatively
irreducible bivariate polynomials than decomposable ones, as one would expect.

\section{Acknowledgements}

I appreciate the interesting discussions with Arnaud Bodin, Pierre
Dèbes, and Salah Najib about the topic, and in particular the
challenges that the preliminary version of \cite{boddeb09} posed.

A first version of this paper is in \cite{gat08b}. The work was supported
by the B-IT Foundation and the Land Nordrhein-Westfalen.


\bibliographystyle{cc2} \bibliography{journals,refs,lncs}
\end{document}